\documentclass{amsart}
\usepackage{amssymb}
\usepackage[nobysame]{amsrefs}
\usepackage{todonotes}
\usepackage{mathpazo}
\usepackage{enumerate}
\usepackage{subcaption}
\usepackage{tikz,pgf}
	\pgfdeclarelayer{background}
	\pgfdeclarelayer{middle}
	\pgfsetlayers{background,middle,main}
\usepackage[colorlinks=true,
	urlcolor=blue!20!black,
	citecolor=green!50!black]{hyperref}
\newtheorem{theorem}{Theorem}[section]
\newtheorem{proposition}[theorem]{Proposition}
\newtheorem{lemma}[theorem]{Lemma}
\newtheorem{corollary}[theorem]{Corollary}

\theoremstyle{remark}

\newcommand{\defn}[1]{{\color{green!50!black}\emph{#1}}}
\newcommand{\defs}{\stackrel{\mathrm{def}}{=}}

\newcommand{\Lattice}{\mathbf{L}}
\newcommand{\Poset}{\mathbf{P}}
\newcommand{\Bool}{\mathsf{Bool}}
\newcommand{\CLO}{\mathsf{CLO}}
\newcommand{\clorder}{\leq_{\mathsf{clo}}}
\newcommand{\JI}{\mathsf{JoinIrr}}
\newcommand{\canset}{\Gamma}
\newcommand{\CanComplex}{\mathsf{Can}}
\newcommand{\Covers}{\mathcal{E}}

\newcommand{\least}{\hat{0}}
\newcommand{\grtst}{\hat{1}}
\newcommand{\perspective}{\doublebarwedge}
\newcommand{\jsdlabeling}{\lambda_{\mathsf{jsd}}}
\newcommand{\bdef}{\mathsf{bdef}}
\title{Meet-Distributive Lattices have the Intersection Property}
\author{Henri M{\"u}hle}
\address{Technische Universit{\"a}t Dresden, Institut f{\"u}r Algebra, Zellescher Weg 12--14, 01069 Dresden, Germany.}
\email{henri.muehle@tu-dresden.de}
\keywords{meet-distributive lattices, congruence-uniform lattices, canonical join complex, core label order, intersection property}
\subjclass[2010]{06D75}
\thanks{The author has received funding from the European Research Council (Grant Agreement no. 681988, CSP-Infinity).}

\begin{document}

\begin{abstract}
	Meet-distributive lattices form an intriguing class of lattices, because they are precisely the lattices obtainable from a closure operator with the so-called anti-exchange property.  Moreover, meet-distributive lattices are join semidistributive.  Therefore, they admit two natural, secondary structures: the core label order is an alternative order on the lattice elements and the canonical join complex is the flag-simplicial complex on canonical join representations.  In this article we present a characterization of finite meet-distributive lattices in terms of the core label order and the canonical join complex, and we show that the core label order of a finite meet-distributive lattice is always a meet-semilattice.
\end{abstract}

\maketitle

\section{Introduction}
	\label{sec:introduction}
A lattice $\Lattice$ is join semidistributive if every element admits a canonical expression as a join of join-irreducible elements~\cites{whitman41free,whitman42free}.  Consequently, the word problem can be solved efficiently in these lattices.  The set of canonical join representations of a lattice forms a simplicial complex~\cite{reading15noncrossing}*{Proposition~2.2}; the \defn{canonical join complex} of $\Lattice$.  If $\Lattice$ is join semidistributive, then the faces of the canonical join complex are naturally indexed by the elements of $\Lattice$.

Moreover, when $\Lattice$ is join semidistributive, canonical join representations can be computed easily with the help of a certain edge-labeling which is determined by a perspectivity relation~\cite{barnard19canonical}.  This labeling is essentially unique and can be used to define an alternative partial order on $\Lattice$; the \defn{core label order}.

This order first appeared N.~Reading's research on congruence-uniform lattices of regions of real hyperplane arrangements.  We have investigated this order abstractly for congruence-uniform lattices in \cite{muehle19the}.  For some special cases the core label order was studied in \cites{bancroft11the,clifton18canonical,garver17enumerative,garver18oriented,muehle21noncrossing,petersen13on,reading11noncrossing}.  

An interesting subclass of join-semidistributive lattices are \defn{meet-distributive} lattices, which have the property that every interval $[x,y]$---where $x$ is the meet of the elements covered by $y$---is isomorphic to a Boolean lattice~\cites{dilworth40lattices,edelman80meet}.  It turns out that we can use the core label order and the canonical join complex to characterize meet-distributive lattices.
	
\begin{theorem}\label{thm:meet_distributive_characterization_face_poset}
	A finite join-semidistributive lattice $\Lattice$ is meet-distributive if and only if $\CLO(\Lattice)$ is the face poset of the canonical join complex of $\Lattice$.
\end{theorem}

We want to point out that we can also use the core label order to characterize finite Boolean lattices.  They are precisely the join-semidistributive lattices that are isomorphic to their own core label order~\cite{muehle19the}*{Theorem~1.5}.  Consequently, the canonical join complex of a finite Boolean lattice is a simplex.

In \cite{reading16lattice}*{Problem~9.5}, N.~Reading asked under what conditions the core label order is again a lattice.  In \cite{muehle19the}*{Section~4.2} we found one such property, which we call the \defn{intersection property}.  This property can be used to characterize the join-semidistributive lattices whose core label orders are meet-semilattices~\cite{muehle19the}*{Theorem~4.8}.  We conclude this article with the observation that every meet-distributive lattice has the intersection property.

\begin{theorem}\label{thm:meet_distributive_intersection_property}
	Every finite meet-distributive lattice $\Lattice$ has the intersection property.  Consequently, $\CLO(\Lattice)$ is a meet-semilattice, and it is a lattice if and only if $\Lattice$ is isomorphic to a Boolean lattice.
\end{theorem}
	
We first recall the necessary basic notions in Section~\ref{sec:distributive}.  After that we define the core label order of a lattice in Section~\ref{sec:core_label_order}, and we define the canonical join complex of a join-semidistributive lattice in Section~\ref{sec:canonical_join_complex_jsd}, where we also prove Theorem~\ref{thm:meet_distributive_characterization_face_poset}.  In Section~\ref{sec:intersection_property} we define the intersection property and prove Theorem~\ref{thm:meet_distributive_intersection_property}.
	
\section{Preliminaries}
	\label{sec:distributive}
\subsection{Basic Notions}
	\label{sec:basics}
Let $\Poset=(P,\leq)$ be a partially ordered set (\defn{poset} for short).  The \defn{dual} poset of $\Poset$ is $\Poset^{*}\defs(P,\geq)$.  

An element $x\in P$ is \defn{minimal} in $\Poset$ if $y\leq x$ implies $y=x$ for all $y\in P$.  Dually, $x\in P$ is \defn{maximal} in $\Poset$ if it is minimal in $\Poset^{*}$.

A \defn{cover relation} of $\Poset$ is a pair $(x,y)$ such that $x<y$ and there is no $z\in P$ such that $x<z<y$.  We usually write $x\lessdot y$ for a cover relation, and we denote the set of all cover relations of $\Poset$ by $\Covers(\Poset)$.  Moreover, if $x\lessdot y$, then we call $x$ a \defn{lower cover} of $y$, and $y$ an \defn{upper cover} of $x$.

A \defn{chain} of $\Poset$ is a totally ordered subset of $P$, and it is \defn{saturated} if it can be written as a sequence of cover relations.  A saturated chain is \defn{maximal} if it contains a minimal and a maximal element of $\Poset$.

We say that $\Poset$ is a \defn{lattice} if for every two elements $x,y\in P$ there exists a greatest lower bound $x\wedge y$ (the \defn{meet}) and a least upper bound $x\vee y$ (the \defn{join}).  Observe that every finite lattice has a unique minimal element (denoted by $\least$) and a unique maximal element (denoted by $\grtst$).

A lattice is \defn{Boolean} if it is isomorphic to the family of subsets of some set $M$ ordered by inclusion.  If $\lvert M\rvert=n$, then we write $\Bool(n)$ for the Boolean lattice with $2^{n}$ elements.

\subsection{Join-Semidistributive Lattices}
	\label{sec:join_semidistributive_lattices}
Let $\Lattice=(L,\leq)$ be a lattice. A \defn{join representation} of $x\in L$ is a set $X\subseteq L$ with $x=\bigvee X$.  A join representation $X$ of $x$ \defn{join-refines} a join representation $X'$ of $x$ if for every $y\in X$ there exists some $y'\in X'$ such that $y\leq y'$.  A join representation $X$ of $x$ is \defn{irredundant} if no proper subset of $X$ joins to $x$, and it is \defn{canonical} if it join-refines every other join representation of $x$.  We denote the canonical join representation of $x\in L$ by $\canset(x)$ (if it exists). 

It turns out that the finite lattices in which every element admits a canonical join representation can be characterized algebraically.  A lattice $\Lattice=(L,\leq)$ is \defn{join semidistributive} if for all $x,y,z\in L$ the following implication holds:
\begin{equation}\label{eq:jsd}\tag{JSD}
	x\vee y=x\vee z\quad\text{implies}\quad x\vee y=x\vee(y\wedge z).
\end{equation}

\begin{theorem}[\cite{freese95free}*{Theorem~2.24}]\label{thm:join_semidistributive_characterization}
	A finite lattice is join semidistributive if and only if every element admits a canonical join representation.
\end{theorem}

\subsection{Meet-Distributive Lattices}
	\label{sec:meet_distributive_lattices}
We now move to a subfamily of the join-semi\-distributive lattices.  Let $\Lattice=(L,\leq)$ be a lattice.  For $x\in L$, we define its \defn{nucleus} to be
\begin{displaymath}
	x_{\downarrow} \defs x\wedge\bigwedge_{y\in L\colon y\lessdot x}{y}.
\end{displaymath}
We call the interval $[x_{\downarrow},x]$ the \defn{core} of $x$.  Then, $\Lattice$ is \defn{meet distributive} if for every $x\in L$, the core $[x_{\downarrow},x]$ is isomorphic to a Boolean lattice.  Figure~\ref{fig:jsd_lattice} shows a join-semidistributive lattice that is not meet distributive, and Figure~\ref{fig:md_lattice} shows a meet-distributive lattice.

\begin{figure}
	\centering
	\begin{subfigure}[t]{.45\textwidth}
		\centering
		\begin{tikzpicture}\small
			\def\x{1};
			\def\y{1};
			\def\s{.6};
			\draw(2*\x,1*\y) node[fill,circle,scale=.6](n1){};
			\draw(1*\x,2*\y) node[fill,circle,scale=.6](n2){};
			\draw(3*\x,2.5*\y) node[fill,circle,scale=.6](n3){};
			\draw(1*\x,3*\y) node[fill,circle,scale=.6](n4){};
			\draw(2*\x,4*\y) node[fill,circle,scale=.6](n5){};
			\draw(n1) -- (n2) node[midway,fill=white,inner sep=.2pt,scale=.9]{$1$};
			\draw(n1) -- (n3) node[midway,fill=white,inner sep=.2pt,scale=.9]{$2$};
			\draw(n2) -- (n4) node[midway,fill=white,inner sep=.2pt,scale=.9]{$3$};
			\draw(n3) -- (n5) node[midway,fill=white,inner sep=.2pt,scale=.9]{$1$};
			\draw(n4) -- (n5) node[midway,fill=white,inner sep=.2pt,scale=.9]{$2$};
		\end{tikzpicture}
		\caption{A join-semidistributive lattice.}
		\label{fig:jsd_lattice}
	\end{subfigure}
	\hspace*{1cm}
	\begin{subfigure}[t]{.45\textwidth}
		\centering
		\begin{tikzpicture}\small
			\def\x{1};
			\def\y{1};
			\def\s{.6};
			\draw(2*\x,1*\y) node(n1){$\emptyset$};
			\draw(1*\x,2*\y) node(n2){$\{1\}$};
			\draw(2*\x,2*\y) node(n3){$\{2\}$};
			\draw(3*\x,2*\y) node(n4){$\{3\}$};
			\draw(2*\x,3*\y) node(n5){$\{1,2,3\}$};
			\draw(n1) -- (n2);
			\draw(n1) -- (n3);
			\draw(n1) -- (n4);
			\draw(n2) -- (n5);
			\draw(n3) -- (n5);
			\draw(n4) -- (n5);
		\end{tikzpicture}
		\caption{The core label order of the lattice from Figure~\ref{fig:jsd_lattice}.}
		\label{fig:jsd_lattice_clo}
	\end{subfigure}
	\caption{A join-semidistributive lattice and its core label order.}
	\label{fig:jsd_lattices}
\end{figure}

\begin{figure}
	\centering
	\begin{subfigure}[t]{.45\textwidth}
		\centering
		\begin{tikzpicture}\small
			\def\x{1};
			\def\y{1};
			\def\s{.6};
			\draw(2*\x,1*\y) node[fill,circle,scale=.6](n1){};
			\draw(1*\x,2*\y) node[fill,circle,scale=.6](n2){};
			\draw(2*\x,2*\y) node[fill,circle,scale=.6](n3){};
			\draw(3*\x,2*\y) node[fill,circle,scale=.6](n4){};
			\draw(1.5*\x,3*\y) node[fill,circle,scale=.6](n5){};
			\draw(2.5*\x,3*\y) node[fill,circle,scale=.6](n6){};
			\draw(2*\x,4*\y) node[fill,circle,scale=.6](n7){};
			\draw(n1) -- (n2) node[midway,fill=white,inner sep=.2pt,scale=.9]{$1$};
			\draw(n1) -- (n3) node[midway,fill=white,inner sep=.2pt,scale=.9]{$2$};
			\draw(n1) -- (n4) node[midway,fill=white,inner sep=.2pt,scale=.9]{$3$};
			\draw(n2) -- (n5) node[midway,fill=white,inner sep=.2pt,scale=.9]{$2$};
			\draw(n3) -- (n5) node[midway,fill=white,inner sep=.2pt,scale=.9]{$1$};
			\draw(n3) -- (n6) node[midway,fill=white,inner sep=.2pt,scale=.9]{$3$};
			\draw(n4) -- (n6) node[midway,fill=white,inner sep=.2pt,scale=.9]{$2$};
			\draw(n5) -- (n7) node[midway,fill=white,inner sep=.2pt,scale=.9]{$3$};
			\draw(n6) -- (n7) node[midway,fill=white,inner sep=.2pt,scale=.9]{$1$};
		\end{tikzpicture}
		\caption{A meet-distributive lattice.}
		\label{fig:md_lattice}
	\end{subfigure}
	\hspace*{1cm}
	\begin{subfigure}[t]{.45\textwidth}
		\centering
		\begin{tikzpicture}\small
			\def\x{1};
			\def\y{1};
			\def\s{.6};
			\draw(2*\x,1*\y) node(n1){$\emptyset$};
			\draw(1*\x,2*\y) node(n2){$\{1\}$};
			\draw(2*\x,2*\y) node(n3){$\{2\}$};
			\draw(3*\x,2*\y) node(n4){$\{3\}$};
			\draw(1*\x,3*\y) node(n5){$\{1,2\}$};
			\draw(2*\x,3*\y) node(n6){$\{1,3\}$};
			\draw(3*\x,3*\y) node(n7){$\{2,3\}$};
			\draw(n1) -- (n2);
			\draw(n1) -- (n3);
			\draw(n1) -- (n4);
			\draw(n2) -- (n5);
			\draw(n2) -- (n6);
			\draw(n3) -- (n5);
			\draw(n3) -- (n7);
			\draw(n4) -- (n6);
			\draw(n4) -- (n7);
		\end{tikzpicture}
		\caption{The core label order of the lattice from Figure~\ref{fig:md_lattice}.}
		\label{fig:md_lattice_clo}
	\end{subfigure}
	\caption{A meet-distributive lattice and its core label order.}
	\label{fig:md_lattices}
\end{figure}

The following result characterizes meet-distributive lattices.  Recall that $\Lattice$ is \defn{lower semimodular} if for all $x,y\in L$ whenever $x,y\lessdot x\vee y$, then $x\wedge y\lessdot x,y$.

\begin{theorem}[\cite{adaricheva03join}*{Theorem~1.9}]\label{thm:meet_distributive_characterization}
	A finite lattice is meet distributive if and only if it is join semidistributive and lower semimodular.
\end{theorem}

Meet-distributive lattices are precisely the lattices that arise from a closure operator satisfying the so-called \emph{anti-exchange property}, see \cites{adaricheva03join,armstrong09sorting,edelman80meet}.

\section{The Core Label Order of a Join-Semidistributive Lattice}
	\label{sec:core_label_order_jsd}
\subsection{The Core Label Order}
	\label{sec:core_label_order}
Motivated by the study of the poset of regions of real hyperplane arrangements, N.~Reading introduced an alternate way to order the elements of a congruence-uniform lattice~\cite{reading16lattice}*{Section~9-7.4}.  In fact, we may generalize this construction to arbitrary, finite lattices.

Let $\Lattice=(L,\leq)$ be a finite lattice, let $M$ be a set and let $\lambda\colon\Covers(\Lattice)\to M$ be an \defn{edge labeling} of $\Lattice$.  The \defn{core label set} of $x\in L$ (with respect to $\lambda$) is
\begin{displaymath}
	\Psi_{\lambda}(x) \defs \bigl\{\lambda(u,v)\mid x_{\downarrow}\leq u\lessdot v\leq x\bigr\}.
\end{displaymath}
We may now define $x\clorder y$ if and only if $\Psi_{\lambda}(x)\subseteq\Psi_{\lambda}(y)$.  In general, this results in a quasi-ordered set $\CLO_{\lambda}(\Lattice)\defs(L,\clorder)$.

We say that $\lambda$ is a \defn{core labeling} if the assignment $x\mapsto\Psi_{\lambda}(x)$ is \defn{injective}.  If $\lambda$ is a core labeling, then it is quickly checked that $\CLO(\Lattice)$ is in fact a partial order; the \defn{core label order}.  Figures~\ref{fig:jsd_lattices} and \ref{fig:md_lattices} illustrate this construction.

\subsection{A Perspectivity Labeling}
	\label{sec:perspectivity_labeling}
Two cover relations $(x_{1},y_{1}),(x_{2},y_{2})\in\Covers(\Lattice)$ are \defn{perspective} if either $y_{1}\vee x_{2}=y_{2}$ and $y_{1}\wedge x_{2}=x_{1}$ or $y_{2}\vee x_{1}=y_{1}$ and $y_{2}\wedge x_{1}=x_{2}$.  We write $(x_{1},y_{1})\perspective (x_{2},y_{2})$ in this case.  This definition is illustrated in Figure~\ref{fig:perspectivity}.

\begin{figure}
	\centering
	\begin{tikzpicture}\small
		\def\x{1};
		\def\y{1};
		\def\s{.5};
		\draw(2*\x,1*\y) node[fill,circle,scale=.6](n1){};
		\draw(1*\x,2*\y) node[fill,circle,scale=.6](n2){};
		\draw(3*\x,2*\y) node[fill,circle,scale=.6](n3){};
		\draw(1*\x,3*\y) node[fill,circle,scale=.6](n4){};
		\draw(3*\x,3*\y) node[fill,circle,scale=.6](n5){};
		\draw(2*\x,4*\y) node[fill,circle,scale=.6](n6){};
		\draw[thick,draw=green!50!black](n1) -- (n2);
		\draw(n1) -- (n3);
		\draw(n2) -- (n4);
		\draw(n3) -- (n5);
		\draw(n4) -- (n6);
		\draw[thick,draw=green!50!black](n5) -- (n6);
		\draw[white!50!gray,<->,dashed](1.55*\x,1.55*\x) -- (2.45*\x,3.45*\x);
	\end{tikzpicture}
	\caption{The green edges represent perspective cover relations.}
	\label{fig:perspectivity}
\end{figure}

Recall another useful fact about join-semidistributive lattices.  An element $j\in L$ is \defn{join irreducible} if whenever $j=x\vee y$, then $j\in\{x,y\}$.  The set of join-irreducible elements of $\Lattice$ is denoted by $\JI(\Lattice)$.  In particular, if $\Lattice$ is finite and $j\in\JI(\Lattice)$, then there exists a unique element $j_{*}\in L$ such that $(j_{*},j)\in\Covers(\Lattice)$.

\begin{lemma}[\cite{adaricheva03join}*{Lemma~1.8}]\label{lem:jsd_labeling}
	Let $\Lattice$ be a finite join-semidistributive lattice.  For $(x,y)\in\Covers(\Lattice)$, the set $\{z\in L\mid z\leq y\;\text{and}\;z\not\leq x\}$ has a unique minimal element $j$, and $j$ is join irreducible.
\end{lemma}

This gives rise to the following edge-labeling of a finite, join-semidistributive lattice $\Lattice$:
\begin{equation}\label{eq:jsd_labeling}
	\jsdlabeling\colon\Covers(\Lattice)\to\JI(\Lattice), \quad (x,y)\mapsto\bigwedge\{z\in L\mid z\leq y\;\text{and}\;z\not\leq x\}.
\end{equation}
This labeling is illustrated in Figures~\ref{fig:jsd_lattice} and \ref{fig:md_lattice}.

\begin{lemma}\label{lem:perspective_irreducibles}
	Let $(x,y)\in\Covers(\Lattice)$ and $j\in\JI(\Lattice)$.  If $(x,y)\perspective(j_{*},j)$, then $j\leq y$.
\end{lemma}
\begin{proof}
	If $(x,y)\perspective(j_{*},j)$, then either $j\leq y$ or $y\leq j$.  The latter case, however, forces the existence of two lower covers of $j$, contradicting that $j$ is join irreducible.
\end{proof}

We now show that the labeling $\jsdlabeling$ is a \emph{canonical} labeling of a finite, join-semi\-distributive lattice, because it is determined by the perspectivity relation.

\begin{lemma}\label{lem:jsd_labeling_perspective}
	Let $(x,y)\in\Covers(\Lattice)$.  Then $\jsdlabeling(x,y)=j$ if and only if $(j_{*},j)\perspective(x,y)$.  
\end{lemma}
\begin{proof}
	Suppose that $\jsdlabeling(x,y)=j$.  By definition, $x\vee j=y$ and thus $x\wedge j<j$.  Since $j$ is minimal with the property that $j\leq y$ and $j\not\leq x$, we see that $j_{*}\leq x$.  This implies $x\wedge j=j_{*}$, and it follows that $(j_{*},j)\perspective(x,y)$.
	
	Conversely, suppose that $(j_{*},j)\perspective(x,y)$.  By Lemma~\ref{lem:perspective_irreducibles}, we get $j\vee x=y$ and $j\wedge x=j_{*}$.  Thus, $\jsdlabeling(x,y)\leq j$.  But, $j_{*}\leq x$, which means that $\jsdlabeling(x,y)\not\leq j_{*}$.  Since $j\in\JI(\Lattice)$, we must have $\jsdlabeling(x,y)=j$.
\end{proof}

%

The labeling $\jsdlabeling$ also allows for a simple computation of canonical join representations.

\begin{proposition}[\cite{barnard19canonical}*{Lemma~19}]\label{prop:jsd_labeling_canonical}
	If $\Lattice=(L,\leq)$ is a finite, join-semidistributive lattice, then for every $x\in L$:
	\begin{displaymath}
		\canset(x) = \bigl\{\jsdlabeling(x',x)\mid x'\lessdot x\bigr\}.
	\end{displaymath}
\end{proposition}

\begin{proposition}\label{prop:jsd_labeling_core}
	The edge-labeling $\jsdlabeling$ of a finite, join-semidistributive lattice is a core labeling.
\end{proposition}
\begin{proof}
	Let $\Lattice=(L,\leq)$ be a finite, join-semidistributive lattice, and let $x\in L$.  
	
	If $j\in\Psi_{\jsdlabeling}(x)$, then there exist $x_{1},x_{2}\in L$ such that $x_{\downarrow}\leq x_{1}\lessdot x_{2}\leq x$ such that $\jsdlabeling(x_{1},x_{2})=j$.  By Lemma~\ref{lem:jsd_labeling_perspective}, this means that $(j_{*},j)\perspective(x_{1},x_{2})$ and by Lemma~\ref{lem:perspective_irreducibles} it follows that $j\leq x_{2}\leq x$.  As a consequence, $\bigvee\Psi_{\jsdlabeling}(x)\leq x$.  Moreover, by Proposition~\ref{prop:jsd_labeling_canonical}, we have $\canset(x)\subseteq\Psi_{\jsdlabeling}(x)$, and therefore $x=\bigvee\canset(x)\leq\bigvee\Psi_{\jsdlabeling}(x)$.  It follows that $\bigvee\Psi_{\jsdlabeling}(x)=x$.
	
	Now, if there exist $x,y\in L$ such that $\Psi_{\jsdlabeling}(x)=\Psi_{\jsdlabeling}(y)$, then
	\begin{displaymath}
		x = \bigvee\Psi_{\jsdlabeling}(x) = \bigvee\Psi_{\jsdlabeling}(y) = y.
	\end{displaymath}
	Hence, the assignment $x\mapsto\Psi_{\jsdlabeling}(x)$ is injective, and $\jsdlabeling$ is a core labeling.
\end{proof}

\begin{theorem}\label{thm:md_boolean_defect}
	Let $\Lattice=(L,\leq)$ be a finite, join-semidistributive lattice.  Then we have $\canset(x)=\Psi_{\jsdlabeling}(x)$ for all $x\in L$ if and only if $\Lattice$ is meet distributive.
\end{theorem}
\begin{proof}
	If $\Lattice$ is meet distributive, then every core $[x_{\downarrow},x]$ is isomorphic to a Boolean lattice.  If $\Bool(k)$ is the Boolean lattice with the ground set $M=\{1,2,\ldots,k\}$, then it is easy to verify that $\canset(M)=M=\Psi_{\jsdlabeling}(M)$.  This proves that $\canset(x)=\Psi_{\jsdlabeling}(x)$ for all $x\in L$.

	\medskip
	
	Conversely, suppose that $\Lattice$ is not meet distributive.  By Theorem~\ref{thm:meet_distributive_characterization}, $\Lattice$ is not lower semimodular, which means that there exist two elements $x,y\in L$ such that $x,y\lessdot x\vee y$ and---without loss of generality---$(x\wedge y,x)\notin\Covers(\Lattice)$.  This means that there exists $z\in L$ with $x\wedge y<z\lessdot x$.  Suppose that $\jsdlabeling(z,x)=j$.   By construction, $j\in\Psi_{\jsdlabeling}(x\vee y)$.  By perspectivity, $j\neq\jsdlabeling(x,x\vee y)$.
	
	Since $j\leq x$ and $j\not\leq z$, the assumption $x\wedge y<z$ implies that $j\not\leq y$.  Moreover, $z\not\leq y$ because otherwise $z=x\wedge y$.  This implies that $j\vee y=x\vee y=z\vee y$ and by \eqref{eq:jsd} we get $y\neq x\vee y=y\vee(z\wedge j)=y\vee j_{*}$.  Thus $j_{*}\not\leq y$ and since $j\in\JI(\Lattice)$, we find $y\wedge j\neq j_{*}$.  It follows that $\jsdlabeling(y,x\vee y)\neq j$.
	
	If $x,y$ are the only lower covers of $x\vee y$, then we have just shown that $j\notin\canset(x\vee y)$, which yields $\canset(x\vee y)\subsetneq\Psi_{\jsdlabeling}(x\vee y)$.  
	
	Suppose that there exists another lower cover $u$ of $x\vee y$ (different from $x$ and $y$).  If $z\leq u$, then we get $x\vee y=x\vee u=y\vee u$ and therefore by \eqref{eq:jsd} $x\vee y=u\vee(x\wedge y)\leq u\vee z=u$; which is a contradiction.  If $j\leq u$, then $\jsdlabeling(u,x\vee y)\neq j$ by Lemma~\ref{lem:jsd_labeling_perspective}.  Otherwise, we get $j\vee u=x\vee y=z\vee u$ and therefore $u\neq x\vee y=u\vee(z\wedge j)=u\vee j_{*}$, and this implies $\jsdlabeling(u,x\vee y)\neq j$.  Since $u$ was chosen arbitrarily, we conclude that $j\notin\canset(x\vee y)$, and we find that $\canset(x\vee y)\subsetneq\Psi_{\jsdlabeling}(x\vee y)$.
\end{proof}

We may define the \defn{Boolean defect} of a join-semidistributive lattice $\Lattice=(L,\leq)$ by
\begin{displaymath}
	\bdef(\Lattice) \defs \sum_{x\in L}\bigl\lvert\Psi_{\jsdlabeling}(x)\setminus\canset(x)\bigr\rvert.
\end{displaymath}
Theorem~\ref{thm:md_boolean_defect} has the following consequence, which strengthens \cite{muehle19the}*{Proposition~5.2}.

\begin{corollary}
	A finite join-semidistributive lattice $\Lattice$ has $\bdef(\Lattice)=0$ if and only if $\Lattice$ is meet distributive.
\end{corollary}

\subsection{The Canonical Join Complex of a Join-Semidistributive Lattice}
	\label{sec:canonical_join_complex_jsd}
Given a finite set $M$, a \defn{simplicial complex} on $M$ is a family $\Delta(M)$ of subsets of $M$, such that for every $F\in\Delta(M)$ and every $F'\subseteq F$ we have $F'\in\Delta(M)$.  The members of $\Delta(M)$ are \defn{faces}.  The \defn{face poset} of $\Delta(M)$ is the poset $\bigl(\Delta(M),\subseteq\bigr)$.  

N.~Reading has observed in \cite{reading15noncrossing}*{Proposition~2.2} that the set of canonical join representations of a lattice is closed under taking subsets.  In other words, it forms a simplicial complex; the \defn{canonical join complex} of $\Lattice$, denoted by $\CanComplex(\Lattice)$.  

We are now ready to prove Theorem~\ref{thm:meet_distributive_characterization_face_poset}.

\begin{proof}[Proof of Theorem~\ref{thm:meet_distributive_characterization_face_poset}]
	Let $\Lattice=(L,\leq)$ be a finite, join-semidistributive lattice.  By definition, the face poset of $\CanComplex(\Lattice)$ is precisely $\Bigl(\bigl\{\canset(x)\mid x\in L\bigr\},\subseteq\Bigr)$, and $\CLO_{\jsdlabeling}(\Lattice)$ is isomorphic to $\Bigl(\bigl\{\Psi(x)\mid x\in L\bigr\},\subseteq\Bigr)$.
	
	If $\Lattice$ is meet distributive, then Theorem~\ref{thm:md_boolean_defect} states that these two posets are isomorphic.  
	
	If $\Lattice$ is not meet distributive, then by Theorem~\ref{thm:md_boolean_defect}, there exists some $x\in L$ such that $\canset(x)\subsetneq\Psi_{\jsdlabeling}(x)$.  In particular, there exists $j\in\Psi_{\jsdlabeling}(x)\setminus\canset(x)$.  It follows that $\{j\}\subseteq\Psi_{\jsdlabeling}(x)$, but $\{j\}\not\subseteq\canset(x)$, so that the core label order of $\Lattice$ is not isomorphic to the face poset of $\CanComplex(\Lattice)$.
\end{proof}

Figure~\ref{fig:main_theorem} illustrates Theorem~\ref{thm:meet_distributive_characterization_face_poset} on a bigger example.

\begin{figure}
	\centering
	\begin{subfigure}[t]{.2\textwidth}
		\centering
		\begin{tikzpicture}\small
			\def\x{1};
			\def\y{1};
			\def\s{.6};
			\draw(2*\x,1*\y) node[fill,circle,scale=.6](n1){};
			\draw(1*\x,2*\y) node[fill,circle,scale=.6](n2){};
			\draw(2*\x,2*\y) node[fill,circle,scale=.6](n3){};
			\draw(3*\x,2*\y) node[fill,circle,scale=.6](n4){};
			\draw(1*\x,3*\y) node[fill,circle,scale=.6](n5){};
			\draw(2*\x,3*\y) node[fill,circle,scale=.6](n6){};
			\draw(3*\x,3*\y) node[fill,circle,scale=.6](n7){};
			\draw(1*\x,4*\y) node[fill,circle,scale=.6](n8){};
			\draw(2*\x,4*\y) node[fill,circle,scale=.6](n9){};
			\draw(3*\x,4*\y) node[fill,circle,scale=.6](n10){};
			\draw(2*\x,5*\y) node[fill,circle,scale=.6](n11){};
			\draw(n1) -- (n2) node[midway,fill=white,inner sep=.2pt,scale=.9]{$1$};
			\draw(n1) -- (n3) node[midway,fill=white,inner sep=.2pt,scale=.9]{$2$};
			\draw(n1) -- (n4) node[midway,fill=white,inner sep=.2pt,scale=.9]{$3$};
			\draw(n2) -- (n5) node[midway,fill=white,inner sep=.2pt,scale=.9]{$2$};
			\draw(n3) -- (n5) node[midway,fill=white,inner sep=.2pt,scale=.9]{$1$};
			\draw(n3) -- (n6) node[midway,fill=white,inner sep=.2pt,scale=.9]{$4$};
			\draw(n3) -- (n7) node[midway,fill=white,inner sep=.2pt,scale=.9]{$3$};
			\draw(n4) -- (n7) node[midway,fill=white,inner sep=.2pt,scale=.9]{$2$};
			\draw(n5) -- (n8) node[midway,fill=white,inner sep=.2pt,scale=.9]{$4$};
			\draw(n5) -- (n9) node[fill=white,inner sep=.2pt,scale=.9] at (1.25*\x,3.25*\y) {$3$};
			\draw(n6) -- (n8) node[fill=white,inner sep=.2pt,scale=.9] at (1.75*\x,3.25*\y) {$1$};
			\draw(n6) -- (n10) node[fill=white,inner sep=.2pt,scale=.9] at (2.25*\x,3.25*\y) {$3$};
			\draw(n7) -- (n9) node[fill=white,inner sep=.2pt,scale=.9] at (2.75*\x,3.25*\y) {$1$};
			\draw(n7) -- (n10) node[midway,fill=white,inner sep=.2pt,scale=.9]{$4$};
			\draw(n8) -- (n11) node[midway,fill=white,inner sep=.2pt,scale=.9]{$3$};
			\draw(n9) -- (n11) node[midway,fill=white,inner sep=.2pt,scale=.9]{$4$};
			\draw(n10) -- (n11) node[midway,fill=white,inner sep=.2pt,scale=.9]{$1$};
		\end{tikzpicture}
		\caption{A meet-distributive lattice.}
		\label{fig:md_lattice_2}
	\end{subfigure}
	\hspace*{.25cm}
	\begin{subfigure}[t]{.4\textwidth}
		\centering
		\begin{tikzpicture}\small
			\def\x{1};
			\def\y{1};
			\def\s{.6};
			\draw(2.5*\x,1*\y) node(n1){$\emptyset$};
			\draw(1*\x,2*\y) node(n2){$\{2\}$};
			\draw(2*\x,2*\y) node(n3){$\{1\}$};
			\draw(3*\x,2*\y) node(n4){$\{3\}$};
			\draw(4*\x,2*\y) node(n5){$\{4\}$};
			\draw(.5*\x,3*\y) node(n6){$\{1,2\}$};
			\draw(1.5*\x,3*\y) node(n7){$\{2,3\}$};
			\draw(2.5*\x,3*\y) node(n8){$\{1,3\}$};
			\draw(3.5*\x,3*\y) node(n9){$\{1,4\}$};
			\draw(4.5*\x,3*\y) node(n10){$\{3,4\}$};
			\draw(3.5*\x,4*\y) node(n11){$\{1,3,4\}$};
			\draw(n1) -- (n2);
			\draw(n1) -- (n3);
			\draw(n1) -- (n4);
			\draw(n1) -- (n5);
			\draw(n2) -- (n6);
			\draw(n2) -- (n7);
			\draw(n3) -- (n6);
			\draw(n3) -- (n8);
			\draw(n3) -- (n9);
			\draw(n4) -- (n7);
			\draw(n4) -- (n8);
			\draw(n4) -- (n10);
			\draw(n5) -- (n9);
			\draw(n5) -- (n10);
			\draw(n8) -- (n11);
			\draw(n9) -- (n11);
			\draw(n10) -- (n11);
		\end{tikzpicture}
		\caption{The core label order of the lattice from Figure~\ref{fig:md_lattice_2}.}
		\label{fig:md_lattice_2_clo}
	\end{subfigure}
	\hspace*{.25cm}
	\begin{subfigure}[t]{.3\textwidth}
		\centering
		\begin{tikzpicture}\small
			\def\x{1.5};
			\def\y{1.5};
			\draw(1*\x,1*\y) node(n4){$4$};
			\draw(2*\x,1*\y) node(n3){$3$};
			\draw(1.5*\x,2*\y) node(n1){$1$};
			\draw(2.5*\x,2*\y) node(n2){$2$};
			\draw(n1) -- (n2);
			\draw(n1) -- (n3);
			\draw(n1) -- (n4);
			\draw(n2) -- (n3);
			\draw(n3) -- (n4);
			\begin{pgfonlayer}{background}
				\fill[gray!50!white](1*\x,1*\y) -- (2*\x,1*\y) -- (1.5*\x,2*\y) -- cycle;
			\end{pgfonlayer}
		\end{tikzpicture}
		\caption{The canonical join complex of the lattice from Figure~\ref{fig:md_lattice_2}.  The highlighted region indicates a two-dimensional face.}
		\label{fig:md_lattice_2_can}
	\end{subfigure}
	\caption{Illustration of Theorem~\ref{thm:meet_distributive_characterization_face_poset}.}
	\label{fig:main_theorem}
\end{figure}

\subsection{The Intersection Property}
	\label{sec:intersection_property}
N.~Reading asked in \cite{reading16lattice}*{Problem~9.5} for conditions on a congruence-uniform lattice $\Lattice$ which would imply that $\CLO(\Lattice)$ is a lattice, too.  We gave one such property in \cite{muehle19the}*{Section~4.2}, which extends to arbitrary lattices as follows.  A finite $\Lattice=(L,\leq)$ with edge labeling $\lambda$ \defn{has the intersection property} if for all $x,y\in L$ there exists $z\in L$ such that $\Psi_{\lambda}(x)\cap\Psi_{\lambda}(y)=\Psi_{\lambda}(z)$.  

Provided that $\lambda$ is a core labeling, the proof of \cite{muehle19the}*{Theorems~1.3~and~4.7} carries over essentially verbatim to the more general case.

\begin{theorem}[\cite{muehle19the}*{Theorems~1.3~and~4.7}]\label{thm:intersection_property}
	Let $\Lattice$ be a finite lattice with core labeling $\lambda$.  The core label order $\CLO_{\lambda}(\Lattice)$ is a meet-semilattice if and only if $\Lattice$ has the intersection property.  It is a lattice if and only if $\grtst_{\downarrow}=\least$.
\end{theorem}

We conclude this article with the proof of Theorem~\ref{thm:meet_distributive_intersection_property}.

\begin{proof}[Proof of Theorem~\ref{thm:meet_distributive_intersection_property}]
	Let $\Lattice=(L,\leq)$ be a finite meet-distributive lattice.  For $x,y\in L$ we conclude from Theorem~\ref{thm:md_boolean_defect} that $\Psi_{\jsdlabeling}(x)=\canset(x)$ and $\Psi_{\jsdlabeling}(y)=\canset(y)$.  It follows that $Z=\canset(x)\cap\canset(y)$ is a face of $\CanComplex(\Lattice)$, which means that there exists $z\in L$ with $Z=\canset(z)=\Psi_{\jsdlabeling}(z)$.  We have thus established that $\Lattice$ has the intersection property.  
	
	Lemma~3.9 of \cite{muehle19the} states that $\CLO(\Lattice)$ has a greatest element if and only if $\grtst_{\downarrow}=\least$.  Now, if $\Lattice$ is meet-distributive, then the interval $[\grtst_{\downarrow},\grtst]$ is isomorphic to a Boolean lattice.  Thus, $\grtst_{\downarrow}=\least$ if and only if $\Lattice$ is Boolean.  The claims then follows from Theorem~\ref{thm:intersection_property}.
\end{proof}


\end{document}